\documentclass[a4paper,english]{article}
\usepackage{amssymb,latexsym}
\usepackage{amsmath}
\usepackage{epsfig}
\usepackage{graphics}
\newtheorem{theorem}{Theorem}
\newtheorem{acknowledgement}{Acknowledgement}
\newtheorem{corollary}{Corollary}
\newtheorem{example}{Example}
\newenvironment{proof}[1][Proof]{\noindent\textbf{#1.} }{\ \rule{0.5em}{0.5em}}
\numberwithin{equation}{section}
\numberwithin{theorem}{section}
\numberwithin{corollary}{section}
\numberwithin{example}{section}

\begin{document}
\title{Lamarle Formula in 3-Dimensional Lorentz Space }
\author{Soley ERSOY, Murat TOSUN\\
         Department of Mathematics,
Faculty of Arts and Sciences
\\
Sakarya University, 54187 Sakarya/TURKEY \\}

\maketitle

\begin{abstract}
The Lamarle Formula, given by Kruppa in \cite{Kr}, is known as a
relationship between the Gaussian curvature and the distribution
parameter of a ruled surface in the surface theory. The ruled
surfaces were investigated in 3 different classes with respect to
the character of base curves and rulings, \cite{Tu1},\cite{Tu2}.
In this paper on account of these studies, the relationships
between the Gaussian curvatures and distribution parameters of
spacelike ruled surface, timelike ruled surface with spacelike
ruling and timelike ruled surface with timelike ruling are
obtained, respectively. These relationships are called as
Lorentzian Lamarle formulas.
Finally some examples concerning with these relations are given.
\end{abstract}

\section{Introduction}\label{S:intro}
The study of ruled surface in $\mathbb{R}^3$ is classical subject
in differential geometry. It has again been studied in some areas
(i.e. Projective geometry, \cite{Sa}, Computer-aided design,
\cite{Po}, etc.) Also, it is well known that the geometry of ruled
surface is very important of kinematics or spatial mechanisms in $
\mathbb{R}^3$, \cite{Gu}, \cite{Ko}. A ruled surface is one which
can be generated by sweeping a line through space. Developable
surfaces are special cases of ruled surfaces, \cite{On}.
Cylindrical surfaces are examples of developable surfaces. On a
developable surface at least one of the two principal curvatures
is zero at all points.Consequently the Gaussian curvature is zero
everywhere too. So it is meaningful
for us to study non-cylindrical ruled surfaces.\\
Lorentz metrics in $3-$dimensional Lorentz space $\mathbb{R}_1^3$
is indefinite. In the theory of relativity, geometry of indefinite
metric is very crucial. Hence, the theory of ruled surface in
Lorentz space $\mathbb{R} _1^3$, which has the metric $ds^2 = dx_1
^2  + dx_2 ^2  - dx_3 ^2$, attracted much attention. The situation
is much more complicated than the Euclidean case, since the ruled
surfaces may have a definite metric (spacelike surfaces), Lorentz
metric (timelike surfaces) or mixed metric. Some characterizations
for ruled surfaces are obtained by \cite{Iz}. Timelike and
spacelike ruled surfaces are defined and the characterizations of
timelike and spacelike ruled surfaces are found in \cite{Ay},
\cite{Ka1}, \cite{Ka2}, \cite{Tu1} and \cite{Tu2}.

\section{Preliminaries}
Let $\mathbb{R} _1^3$ denote the $3-$dimensional Lorentz space,
i.e. the Euclidean space $E^3$ with standard flat metric given by

\begin{equation}\label{A:2.1}
g = dx_1 ^2  + dx_2 ^2  - dx_3 ^2
\end{equation}
where $\left( {x_1 ,x_2 ,x_3 } \right)$ is rectangular coordinate
system of $\mathbb{R} _1^3$. Since $g$ is indefinite metric,
recall that a vector $\vec v$ in $\mathbb{R }_1^3$ can have one of
three casual characters: it can be space-like if $g\left( {\vec
v,\vec v} \right) > 0$ or $\vec v=0$, time-like if $g\left( {\vec
v,\vec v} \right) < 0$ and null $g\left( {\vec v,\vec v} \right) =
0$ and $\vec v \ne 0$. Similarly, an arbitrary curve $\vec \alpha
= \vec \alpha \left( s \right) \subset\mathbb{R} _1^3$  can
locally be space-like, time-like or null (light-like), if all of
its velocity vectors $\vec \alpha '\left( s \right)$ are
respectively space-like, time-like or null (light-like). The norm
of a vector $\vec v$ is given by $\left\| \vec v \right\| = \sqrt
{\left| {g\left( {\vec v,\vec v} \right)} \right|}$. Therefore,
$\vec v$ is a unit vector if $g\left( {\vec v,\vec v} \right) =
\mp 1$. Furthermore, vectors $\vec v$ and $\vec w$
are said to be orthogonal if $g\left( {\vec v,\vec w} \right) = 0$, \cite{On}.\\
Let the set of all timelike vectors in $\mathbb{R }_1^3$ be
$\Gamma$. For $ \vec u \in \Gamma$, we call
\[
C\left( {\vec u} \right) = \left\{ {\left. {\vec v \in \Gamma }
\right|{\rm  }\left\langle {\vec v,\vec u} \right\rangle  < 0}
\right\}
\]
as time-conic of Lorentz space $\mathbb{R }_1^3$ including
vector $\vec u$, \cite{On}.\\
 Let $\vec v$ and $\vec w$ be two time-like vectors
in Lorentz space $\mathbb{R}_1^3$. In this case there exists the
following inequality
 \[
\left| {g\left( {\vec v,\vec w} \right)} \right| \ge \left\| {\vec
v} \right\|.\left\| {\vec w} \right\|.
\]
In this inequality if one wishes the equality condition, then it
is necessary for $\vec v$ and $\vec w$ be linear dependent.\\
If time-like vectors $\vec v$ and $\vec w$ stay inside the same
time-conic then there is a unique non-negative real number of $
\theta  \ge 0$ such that

\begin{equation}\label{A:2.2}
g\left( {\vec v,\vec w} \right) =  - \left\| {\vec v}
\right\|.\left\| {\vec w} \right\|.\cosh \theta
\end{equation}
where the number $\theta$ is called an angle between the timelike
vectors, \cite{On}.\\
Let $\vec v$ and $\vec w$ be spacelike vectors in $\mathbb{R}_1^3$
that span a spacelike subspace. We have that
\[
\left| {g\left( {\vec v,\vec w} \right)} \right| \le \left\| {\vec
v} \right\|.\left\| {\vec w} \right\|
\]
with equality if and only if $\vec v$ and $\vec w$ are linearly
dependent. Hence, there is a unique angle $0 \le \theta  \le \pi $
such that

\begin{equation}\label{A:2.3}
g\left( {\vec v,\vec w} \right) = \left\| {\vec v}
\right\|.\left\| {\vec w} \right\|.\cos \theta
\end{equation}
where the number $\theta$ is called the Lorentzian spacelike angle
between spacelike vectors  $\vec v$ and $\vec w$, \cite{Ra}.\\
Let $\vec v$ and $\vec w$ be spacelike vectors in
$\mathbb{R}_1^3$ that span a timelike subspace. We have that
\[ \left|
{g\left( {\vec v,\vec w} \right)} \right| > \left\| {\vec v}
\right\|.\left\| {\vec w} \right\|.
\]
Hence, there is a unique real number $ \theta  > 0 $
 such that

\begin{equation}\label{A:2.4}
g\left( {\vec v,\vec w} \right) = \left\| {\vec v}
\right\|.\left\| {\vec w} \right\|.\cosh \theta.
\end{equation}
The Lorentzian timelike angle between spacelike vectors $\vec v$
and $\vec w$ is defined to be $\theta$, \cite{Ra}.\\
Let $\vec v$ be a spacelike vector and $\vec w$ be a timelike
vector in $\mathbb{R}_1^3$. Then there is a unique real number $
\theta  \ge 0$ such that

\begin{equation}\label{A:2.5}
g\left( {\vec v,\vec w} \right) = \left\| {\vec v} \right\|.\left\|
{\vec w} \right\|.\sinh \theta.
\end{equation}
The Lorentzian timelike angle between $\vec v$ and $\vec w$ is
defined to be $\theta$, \cite{Ra}.\\
For any vectors  $\vec v = \left( {v_1 ,v_2 ,v_3 } \right)$, $
\vec w = \left( {w_1 ,w_2 ,w_3 } \right) \in \mathbb{R}_1^3$ , the
Lorentzian product $\vec v \wedge \vec w$ of $\vec v$ and
$\vec w$ is defined as \cite{Ak}\\

 \begin{equation}\label{A:2.6}
\vec v \wedge \vec w = \left( {v_3 w_2  - v_2 w_3 ,v_1 w_3  - v_3
w_1 ,v_1 w_2  - v_2 w_1 } \right).
\end{equation}

\section{Ruled Surface in $\mathbb{R}_1^3$}
A ruled surface $M \in \mathbb{R}_1^3$ is a regular surface that
has a parametrization $\varphi:\left({I \times \mathbb{R}}\right)
\to \mathbb{R}_1^3$ of the form

\begin{equation}\label{A:3.1}
\varphi \left( {u,v} \right) = \vec \alpha \left( u \right) +
v\vec \gamma \left( u \right)
\end{equation}
where $\vec \alpha$ and $\vec \gamma$ are curves in
$\mathbb{R}_1^3$ with $\vec \alpha '$
 never vanishes.
The curve $\alpha$ is called the base curve. The rulings of ruled
surface are the straight lines $ v \to \vec \alpha \left( u
\right) + v\vec \gamma \left( u \right) $.\\
If consecutive rulings of a ruled surface in $\mathbb{R}_1^3$
intersect, then the surface is said to be developable. All other
ruled surfaces are called skew surfaces. If there exists a common
perpendicular to two constructive rulings in the skew surface,
then the foot of the common perpendicular on the main ruling is
called a striction point. The set of striction points on a ruled
surface defines the striction curve,\cite{Tu2}.\\
The striction curve, $\beta \left( u \right)$ can be written in
terms of the base curve $\alpha \left( u \right)$ as
\begin{equation}\label{A:3.2}
\vec \beta \left( u \right) = \vec \alpha \left( u \right) -
\frac{{g\left( {\vec \alpha '\left( u \right),\vec \gamma '\left(
u \right)} \right)}}{{\left\langle {\vec \gamma ',\vec \gamma '}
\right\rangle }}\,\,\vec \gamma \left( u \right).
\end{equation}
A ruled surface given by (\ref{A:3.1}) is called non-cylindrical
if $\vec \gamma  \wedge \vec \gamma '$ is nowhere zero. Thus, the
rulings are always changing directions on a non-cylindrical ruled
surface. A non-cylindrical ruled surface always has a
parameterization of the form
\begin{equation}\label{A:3.3}
\tilde \varphi \left( {u,v} \right) = \vec \beta \left( u \right) +
v\;\vec e\left( u \right)
\end{equation}
where $\left\| {\vec e\left( u \right)} \right\| = \frac{{\vec
\gamma \left( u \right)}}{{\left\| {\vec \gamma \left( u \right)}
\right\|}} = 1$, $\left\langle {\vec \beta '\left( u \right),\vec
e'\left( u \right)} \right\rangle  = 0$ and $\vec \beta \left( u
\right)$ is striction curve of  $\tilde \varphi$, \cite{Gr}.\\
The distribution parameter (or drall) of a non-cylindrical ruled
surface given by equation (\ref{A:3.3}), is a function $P$ defined
by
\begin{equation}\label{A:3.4}
P = \frac{{\det \left( {\vec \beta ',\vec e,\vec e'}
\right)}}{{\left\langle {\vec e',\vec e'} \right\rangle }}.
\end{equation}
where $\vec \beta$ is the striction curve and $\vec e$ is the
director curve. Moreover, Gaussian curvature of non-cylindrical
ruled surface $ \tilde \varphi \left( {u,v} \right)$ is
\begin{equation}\label{A:3.5}
K = \frac{{LM - N^2 }}{{EG - F^2 }}
\end{equation}
where $E$, $F$ and $G$ are the coefficients of the first
fundamental form, whereas $L$, $N$ and $M$ are the coefficients of
the second fundamental form, of non-cylindrical ruled
surface, \cite{Gr}.\\
The unit normal vector of non-cylindrical ruled surface $ \tilde
\varphi$ is given by
\begin{equation}\label{A:3.6}
\eta \left( {u,v} \right) = \frac{{\tilde \varphi _u  \wedge
\tilde \varphi _v }}{{\left\| {\tilde \varphi _u  \wedge \tilde
\varphi _v } \right\|}}.
\end{equation}
A surface in the $3-$dimensional Minkowski space-time
$\mathbb{R}_1^3$ is called a time-like surface if induced metric
on the surface is a Lorentzian metric i.e. the normal on the
surface is a
space-like vector, \cite{Tu2}.\\
In $\mathbb{R}_1^3$, according to the character of the non-null
base curve and the non-null ruling, ruled surfaces are classified
into three different groups. As a spacelike ruling moves along a
spacelike curve it generates a spacelike ruled surface, that will
be denoted by $M_1$. Furthermore, the movement of a timelike
ruling along a spacelike curve and the movement of a spacelike
ruling along a timelike curve generate timelike ruled surfaces.
Let us denote these timelike ruled surfaces by $M_2$ and $M_3$,
respectively. Now, we will establish Lamarle formula for these
ruled surfaces  $M_1$, $M_2$, $M_3$ separately.

\section{Lamarle Formula for the Spacelike Ruled Surface}
Let $M_1$ be a spacelike ruled surface parametrized by
\[
\begin{array}{l}
 \varphi _1 :{\rm I} \times \mathbb{R} \to \mathbb{R}_1^3  \\
 {\rm                    }\left( {u,v} \right) \to \varphi _1 \left( {u,v} \right) = \vec \alpha _1 \left( u \right) + v\;\vec e_1 \left( u \right). \\
 \end{array}
\]
If we choose $ \left\| {\vec e_1 } \right\| = 1$ , $\vec n_1  =
\frac{{\vec e'_1 }}{{\left\| {\vec e'_1 } \right\|}}$ and $\vec \xi
_1  = \frac{{\vec e_1  \wedge \vec e'_1 }}{{\left\| {\vec e_1 \wedge
\vec e'_1 } \right\|}}$, we obtain the orthonormal frame field
$\left\{ {\vec e_1 ,\vec n_1 ,\vec \xi _1 } \right\}$. Suppose that
these orthonormal frame field forms right handed system and is $
\left\{ {{\rm space}{\rm , time}{\rm , space}} \right\}$ type. In
this case we may write
\begin{equation}\label{A:4.1}
\begin{array}{l}
 \left\langle {\vec e_1 ,\vec e_1 } \right\rangle  = 1\;,\;\left\langle {\vec n_1 ,\vec n_1 } \right\rangle  =  - 1\;,\;\left\langle {\vec \xi _1 ,\vec \xi _1 } \right\rangle  = 1, \\
 \left\langle {\vec e_1 ,\vec n_1 } \right\rangle  = \left\langle {\vec n_1 ,\vec \xi _1 } \right\rangle  = \left\langle {\vec \xi _1 ,\vec e_1 } \right\rangle  = 0 \\
 \end{array}
\end{equation}
and

\begin{equation}\label{A:4.2}
\vec e_1  \wedge \vec n_1  =  - \vec \xi _1 \quad ,\quad \vec n_1
\wedge \vec \xi _1  =  - \vec e_1 \quad ,\quad \vec \xi _1  \wedge
\vec e_1  = \vec n_1 .
\end{equation}
The Frenet formulae of this orthonormal frame along $e_1$ become

\begin{equation}\label{A:4.3}
\vec e'_1  = \kappa _1 \,\vec n_1 \quad ,\quad \vec n'_1  = \kappa
_1 \,\vec e_1  + \tau _1 \,\vec \xi _1 \quad ,\quad \vec \xi' _1
= \tau _1 \,\vec n_1 .
\end{equation}
Let $\vec \beta _1 \left( u \right)$ be a striction curve of
spacelike ruled surface $M_1$ given by equation (\ref{A:3.2}) in
$\mathbb{R}_1^3$. In this case the tangent vector $\vec \beta '_1$
of this curve stays in spacelike plane. Taking the angle
$\sigma_1$ to be the angle between $\vec \beta '_1$ and $\vec e_1$
since the tangent vector of striction curve of $M_1$ is
\[
\vec \beta '_1  = \vec e_1 \,\cos \sigma _1  + \vec \xi _1 \,\sin
\sigma _1
\]
we find the striction curve of $M_1$ to be
\[
\vec \beta _1  = \int {\left( {\,\cos \sigma _1 \,\vec e_1  +
\,\sin \sigma _1 \,\vec \xi _1 } \right)\,} du.
\]
The spacelike non-cylindrical ruled surface $M_1$ is parametrized
by
\[
\tilde \varphi _1 \left( {u,v} \right) = \int {\left( {\cos \sigma
_1 \,\vec e_1 \, + \sin \sigma _1 \,\vec \xi _1 } \right)\,} du +
v\,\vec e_1 .
\]
From the equation (\ref{A:3.4}) the distribution parameter of
$M_1$ is found to be
\[
P = \frac{{\det \left( {\cos \sigma _1 {\kern 1pt} \vec e_1 \, +
\sin \sigma _1 {\kern 1pt} \vec \xi _1 \,,\vec e_1 ,\kappa _1 \vec
n_1 \,} \right)}}{{\left\langle {\kappa _1 \vec n_1 \,,\kappa _1
\vec n_1 } \right\rangle }} = \frac{{\sin \sigma _1 }}{{\kappa _1
}}.
\]
Adopting $\kappa _1  = \frac{1}{{\rho _1 }}$ we get the
distribution parameter as follows

\begin{equation}\label{A:4.4}
P = \rho _1 \,\sin \sigma _1 .
\end{equation}
Considering equation (\ref{A:4.2}) from equation (\ref{A:3.6}) we
write the unit normal tangent vector of spacelike non-cylindrical
ruled surface $M_1$

\begin{equation}\label{A:4.5}
\vec \eta _1  = \frac{{\sin \sigma _1 \,\vec n_1 \, + v{\kern 1pt}
\kappa _1 \,\vec \xi _1 {\kern 1pt} }}{{\sqrt {\left| { - \sin ^2
\sigma _1  + v^2 \kappa _1^2 } \right|} }}.
\end{equation}
Taking into consideration that $\kappa _1  = \frac{1}{{\rho _1 }}$
and equation (\ref{A:4.4}) we obtain

\begin{equation}\label{A:4.6}
\vec \eta _1  = \frac{{P\vec n_1  + \,v\vec \xi _1 }}{{\sqrt
{\left| { - P^2  + v^2 } \right|} }}.
\end{equation}
Furthermore, since the unit normal tangent vector $ \eta _1$ of a
spacelike surface $M_1$ is timelike we find that $- P^2  + v^2<0$,
that is $\left| v \right| < \left| P \right|$.\\
The partial differentiation of $M_1$ with respect to $u$ and $v$
from equation (\ref{A:4.3}) are as follows

\begin{equation}\label{A:4.7}
\begin{array}{l}
 \tilde \varphi _{1u}  = \cos \sigma _1 \,\,\vec e_1  + \sin \sigma _1 \,\,\vec \xi _1  + \,v\,\kappa _1 \vec n_1,  \\
 \tilde \varphi _{1v}  = \vec e_1.  \\
 \end{array}
\end{equation}
Therefore, we find the first fundamental form's coefficients of
$M_1$ to be

\begin{equation}\label{A:4.8}
\begin{array}{l}
 E = \left\langle {\tilde \varphi _{1u} ,\tilde \varphi _{1u} } \right\rangle  = \cos ^2 \sigma _1  + \sin ^2 \sigma _1  - v^2 \kappa _1^2  = 1 - v^2 \kappa _1^2 , \\
 F = \left\langle {\tilde \varphi _{1u} ,\tilde \varphi _{1v} } \right\rangle  = \cos \sigma _1 , \\
 G = \left\langle {\tilde \varphi _{1v} ,\tilde \varphi _{1v} } \right\rangle  = 1. \\
 \end{array}
\end{equation}
In addition to these, the second order partial differentials of
$M_1$ are found to be

\[
\begin{array}{l}
 \tilde \varphi _{1uu}  = \left( { - \sigma '_1 \sin \sigma _1  + v\kappa _1^2 } \right)\vec e_1  + \left( {\kappa _1 \cos \sigma _1  + \tau _1 \sin \sigma _1  + v\kappa '_1 } \right)\vec n_1  + \left( {\sigma '_1 \cos \sigma _1  + v\kappa _1 \tau _1 } \right)\vec \xi _1 , \\
 \tilde \varphi _{1uv}  = \kappa _1 \vec n_1 , \\
 \tilde \varphi _{1vv}  = 0. \\
 \end{array}
\]
From equation (\ref{A:4.5}) and the last equations we get the
coefficients of second fundamental of $M_1$  as

\begin{equation}\label{A:4.9}
\begin{array}{l}
 L = \left\langle {\tilde \varphi _{1uu} ,\vec \eta } \right\rangle  = \frac{{ - \kappa _1 \cos \sigma _1 \sin \sigma _1  - \tau _1 \sin ^2 \sigma _1  - v\kappa '_1 \sin \sigma _1  + \sigma '_1 \cos \sigma _1 v\kappa _1  + v^2 \kappa _1^2 \tau _1 }}{{\sqrt {\left| { - \sin ^2 \sigma _1  + v^2 \kappa _1^2 } \right|} }}, \\
 N = \left\langle {\tilde \varphi _{1uv} ,\vec \eta } \right\rangle  = \frac{{ - \kappa _1 \sin \sigma _1 }}{{\sqrt {\left| { - \sin ^2 \sigma _1  + v^2 \kappa _1^2 } \right|} }}, \\
 M = \left\langle {\tilde \varphi _{1vv} ,\vec \eta } \right\rangle  = 0. \\
 \end{array}
\end{equation}
Considering equation (\ref{A:4.8}) and (\ref{A:4.9}) together, we
give the following theorem for the Gaussian curvature of spacelike
ruled surface $M_1$.

\begin{theorem}\label{T:4.1}
Let $M_1$ be spacelike non-cylindrical ruled surface in
$\mathbb{R}_1^3$. The Gaussian curvature of spacelike
non-cylindrical ruled surface $M_1$ is given in terms of its
distribution parameter $P$ by

\begin{equation}\label{A:4.10}
K =  - \frac{{P^2 }}{{\left( {P^2  - v^2 } \right)^2 }}
\end{equation}
where $ \left| v \right| < \left| P \right|$.
\end{theorem}

\begin{proof}
Substituting equations (\ref{A:4.8}) and (\ref{A:4.9}) into
equation (\ref{A:3.5}) and making appropriate simplifications we
find the Gaussian curvature of $M_1$ to be
\[
K =  - \frac{{\kappa _1^2 \sin ^2 \sigma _1 }}{{\left( {\sin ^2
\sigma _1  - v^2 \kappa _1^2 } \right)^2 }}.
\]
Considering $\kappa _1  = \frac{1}{{\rho _1 }}$ and equation
(\ref{A:4.4}) completes the proof.
\end{proof}\\
The relation between Gaussian curvature and the distribution
parameter of $M_1$ given by equation (\ref{A:4.10}) is called
\textbf{Lorentzian Lamarle formula} for the spacelike
non-cylindrical ruled surface $M_1$.\\
The Lorentzian Lamarle formula for the spacelike ruled surface in
$\mathbb{R}_1^3$ is non-positive. Therefore we give the following
corollary.

\begin{corollary}\label{C:4.1}
Let $M_1$ be a spacelike non-cylindrical ruled surface with
distribution parameter $P$ and Gaussian curvature $K$ in
$\mathbb{R}_1^3$.
\begin{enumerate}
    \item Along a ruling the Gaussian curvature $ K\left( {u,v} \right) \to 0$
    as $v \to  \mp \infty$.
    \item $K\left( {u,v} \right)=0$ if and only if $P=0$.
    \item If the distribution parameter is $P$ never vanishes then $K\left( {u,v} \right)$ is
    continuous and when $v=0$ i.e. at the central point on each ruling, $K\left( {u,v} \right)$ assumes
    its maximum value.
\end{enumerate}
\end{corollary}

\begin{example}\label{E:4.1}
In  $3-$dimensional Lorentz space $\mathbb{R}_1^3$ let us define a
non-cylindrical ruled surface as
\[
\varphi \left( {u,v} \right) = \left( { - v\cosh u\,,\,u\,,\, -
v\sinh u} \right)
\]
that is a $2^{nd}$ type helicoid and a spacelike surface where $ -
1 < v < 1$, see Figure 4.1.\\

\hfil\scalebox{1}{\includegraphics{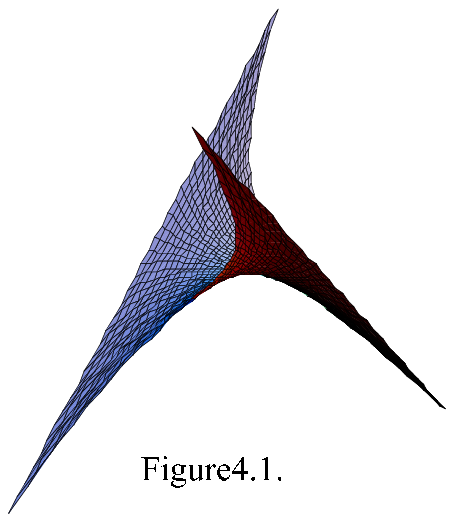}}\hfil
\hfil\scalebox{1}{\includegraphics{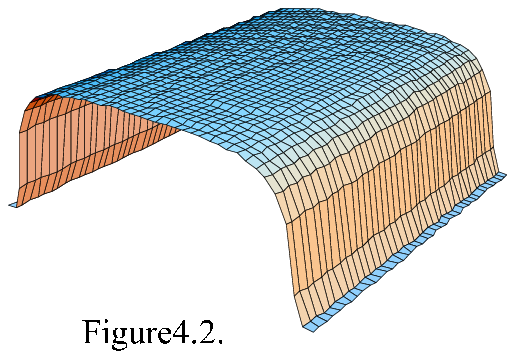}}\hfil
\\The Gaussian curvature of this $2^{nd}$ type helicoid is $
K =  - \frac{1}{{\left( {1 - v^2 } \right)^2 }},\quad \left| v
\right| < 1$, see Figure 4.2.
\end{example}

\section{Lamarle Formula for the Timelike Ruled Surface with Spacelike Base Curve and Timelike Ruling}
Let $M_2$ be timelike ruled surface with spacelike base curve and
timelike ruling in $3-$dimensional Lorentz space, $\mathbb{R
}_1^3$. Thus, this ruled surface is parametrized as follows
\[
\begin{array}{l}
 \varphi _2 :{\rm I} \times \mathbb{R} \to \mathbb{R}_1^3  \\
 {\rm       }\left( {u,v} \right) \to \varphi _2 \left( {u,v} \right) = \vec \alpha _2 \left( u \right) + v\;\vec e_2 \left( u \right). \\
 \end{array}
\]
Here, taking  $ \left\| {\vec e_2 } \right\| = 1$ , $\vec n_2  =
\frac{{\vec e'_2 }}{{\left\| {\vec e'_2 } \right\|}}$ and $\vec \xi
_2  = \frac{{\vec e_2  \wedge \vec e'_2 }}{{\left\| {\vec e_2 \wedge
\vec e'_2 } \right\|}}$, we reach the orthonormal frame field
$\left\{ {\vec e_2 ,\vec n_2 ,\vec \xi _2 } \right\}$.This forms a
right handed system and in $ {\rm \{ time}{\rm , space}{\rm ,
space\} }$ type. Therefore,

\begin{equation}\label{A:5.1}
\begin{array}{l}
  - \left\langle {\vec e_2 ,\vec e_2 } \right\rangle  = \left\langle {\vec n_2 ,\vec n_2 } \right\rangle  = \left\langle {\vec \xi _2 ,\vec \xi _2 } \right\rangle  = 1 \\
 {\rm   }\left\langle {\vec e_2 ,\vec n_2 } \right\rangle  = \left\langle {\vec n_2 ,\vec \xi _2 } \right\rangle  = \left\langle {\vec \xi _2 ,\vec e_2 } \right\rangle  = 0 \\
 \end{array}
\end{equation}
and

\begin{equation}\label{A:5.2}
\vec e_2  \wedge \vec n_2  =  - \vec \xi _2 \quad ,\quad \vec n_2
\wedge \vec \xi _2  =  \vec e_2 \quad ,\quad \vec \xi _2  \wedge
\vec e_2  = -\vec n_2 .
\end{equation}
The differential formulae of this orthonormal system are

\begin{equation}\label{A:5.3}
\vec e'_2  = \kappa _2 \,\vec n_2 \quad ,\quad \vec n'_2  = \kappa
_2 \,\vec e_2  - \tau _2 \,\vec \xi _2 \quad ,\quad \vec \xi' _2
= \tau _2 \,\vec n_2 .
\end{equation}
Now, let the striction curve given by equation (\ref{A:3.2}) of
timelike ruled surface $M_2$ be $\vec \beta _2 \left( u \right)$.
$\vec \beta _2 \left( u \right)$ is a spacelike curve and the
tangent vector of this curve $\vec \beta '_2$ stays in the timelike
plane $\left( {\vec e_2 ,\vec \xi _2 } \right)$. Adopting the
hyperbolic angle $\sigma_2$  between $\vec \beta '_2$ and $\vec e_2$
we write
\[
\vec \beta '_2  = \,\sinh \sigma _2 {\kern 1pt} \vec e_2  + \cosh
\sigma _2 {\kern 1pt} \vec \xi _2 \,.
\]
From the last equation we write for the striction curve of $M_2$
\[
\vec \beta _2  = \int {\left( {\sinh \sigma _2 {\kern 1pt} \vec
e_2  + \cosh \sigma _2 \,\xi _2 } \right)} \,du.
\]
Let $M_2$  be timelike non-cylindrical ruled surface with
spacelike base curve and timelike ruling in $\mathbb{R}_1^3$. In
this case we reparametrize $M_2$ such as
\[
\tilde \varphi _2 \left( {u,v} \right) = \int {\left( {\,\sinh
\sigma _2 {\kern 1pt} \vec e_2  + \,\cosh \sigma _2 {\kern 1pt}
\vec \xi _2 } \right)} \,du + v\,\vec e_2 .
\]
Considering equation (\ref{A:3.4}) we find the distribution
parameter of timelike non-cylindrical ruled surface $M_2$ to be
\[
P = \frac{{\det \left( {\sinh \sigma _2 {\kern 1pt} {\kern 1pt}
\vec e_2 \, + \cosh \sigma _2 \,\vec \xi _2 \,,\vec e_2 ,\kappa _2
\,\vec n_2 } \right)}}{{\left\langle {\kappa _2 \,\vec n_2 ,\kappa
_2 \,\vec n_2 } \right\rangle }} = \frac{{\cosh \sigma _2
}}{{\kappa _2 }}.
\]
Taking  $\kappa _2  = \frac{1}{{\rho _2 }}$ we rewrite the
distribution parameter of $M_2$ as
\begin{equation}\label{A:5.4}
P = \rho _2 \,\cosh \sigma _2 .
\end{equation}
If we consider equation (\ref{A:5.2}), from equation (\ref{A:3.6})
we see that the timelike non-cylindrical ruled surface's unit
normal vector becomes
\begin{equation}\label{A:5.5}
\vec \eta _2  = \frac{{ - \,\cosh \sigma _2 \,\vec n_2  + v{\kern
1pt} \kappa _2 \,\vec \xi _2 {\kern 1pt} }}{{\sqrt {\left| {\cosh
^2 \sigma _2  + v^2 \kappa _2^2 } \right|} }}.
\end{equation}
Since $\kappa _2  = \frac{1}{{\rho _2 }}$, from equation
(\ref{A:5.4}) we find
\begin{equation}\label{A:5.6}
\vec \eta _2  = \frac{{ - {\kern 1pt} P\vec n_2  + v\vec \xi _2
{\kern 1pt} }}{{\sqrt {P^2  + v^2 } }}.
\end{equation}
From equation (\ref{A:5.3}), partial differential of $M_2$ with
respect to $u $ and $v$ are
\begin{equation}\label{A:5.7}
\begin{array}{l}
 \tilde \varphi _{2u}  = \,\sinh \sigma _2 \,\vec e_2  + \,\cosh \sigma _2 \,\vec \xi _2  + {\kern 1pt} v{\kern 1pt} \kappa _2 \vec n_2 , \\
 \tilde \varphi _{2v}  = \vec e_2 . \\
 \end{array}
\end{equation}
Considering the last equations with equation (\ref{A:5.1}) we find
the coefficients of first fundamental form of $M_2$ to be
\begin{equation}\label{A:5.8}
\begin{array}{l}
 E = \left\langle {\tilde \varphi _{2u} ,\tilde \varphi _{2u} } \right\rangle  =  - \sinh ^2 \sigma _2  + \cosh ^2 \sigma _2  + v^2 \,\kappa _2^2  = 1 + v^2 \,\kappa _2^2 , \\
 F = \left\langle {\tilde \varphi _{2u} ,\tilde \varphi _{2v} } \right\rangle  =  - \sinh \sigma _2 , \\
 G = \left\langle {\tilde \varphi _{2v} ,\tilde \varphi _{2v} } \right\rangle  =  - 1. \\
 \end{array}
\end{equation}
Furthermore, if we consider equation, (\ref{A:5.7}) we reach that
the second order partial differentials of $M_2$
\[
\begin{array}{l}
 \tilde \varphi _{2uu}  = \left( {\sigma '_2 \,\cosh \sigma _2  + v{\kern 1pt} \kappa _2^2 } \right)\vec e_2  + \left( {\kappa _2 {\kern 1pt} \sinh \sigma _2  + \tau _2 {\kern 1pt} \cosh \sigma _2  + v{\kern 1pt} \kappa '_2 } \right)\vec n_2  + \left( {\sigma '_2 {\kern 1pt} \sinh \sigma _2  - v{\kern 1pt} \kappa _2 {\kern 1pt} \tau _2 } \right)\vec \xi _2 , \\
 \tilde \varphi _{2uv}  = \kappa _2 \vec n_2 {\kern 1pt} , \\
 \varphi _{2vv}  = 0. \\
 \end{array}
\]
From equation (\ref{A:5.5}) and the last equations we find the
second fundamentals form's coefficients as follows
\begin{equation}\label{A:5.9}
\begin{array}{l}
 L = \left\langle {\tilde \varphi _{2uu} ,\vec \eta _2 } \right\rangle  = \frac{{ - \kappa _2 {\kern 1pt} \sinh \sigma _2 {\kern 1pt} \cosh \sigma _2  - \tau _2 {\kern 1pt} \cosh ^2 \sigma _2  - v{\kern 1pt} \kappa '_2 {\kern 1pt} \cosh \sigma _2  - v{\kern 1pt} \kappa _2 {\kern 1pt} \sigma '_2 {\kern 1pt} \sinh \sigma _2  + v^2 {\kern 1pt} \kappa _2^2 {\kern 1pt} \tau _2 }}{{\sqrt {\cosh ^2 \sigma _2  + v^2 \kappa _2^2 } }}, \\
 N = \left\langle {\tilde \varphi _{2uv} ,\vec \eta _2 } \right\rangle  = \frac{{\kappa _2 \cosh \sigma _2 }}{{\sqrt {\cosh ^2 \sigma _2  + v{\kern 1pt} \kappa _2^2 } }}, \\
 M = \left\langle {\tilde \varphi _{2vv} ,\vec \eta _2 } \right\rangle  = 0. \\
 \end{array}
\end{equation}
Therefore, for the Gaussian curvature of timelike ruled surface
$M_2$, we give the following theorem.
\begin{theorem}\label{T:5.1}
Let $M_2$ be a timelike non-cylindrical ruled surface with
spacelike base curve and timelike ruling in $\mathbb{R}_1^3$.
Taking $P$ to be the distribution parameter of $M_2$ we see that
the Gaussian curvature of $M_2$ is
\begin{equation}\label{A:5.10}
K = \frac{{P^2 }}{{\left( {P^2  + v^2 } \right)^2 }}.
\end{equation}
\end{theorem}
\begin{proof}
Substituting equations (\ref{A:5.8}) and (\ref{A:5.9}) into
equation (\ref{A:3.5}) we find the Gaussian curvature of $M_2$ to
be
\[
K = \frac{{\kappa _2^2 \,\cosh ^2 \sigma _2 }}{{\left( {\cosh ^2
\sigma _2  + \kappa _2^2 {\kern 1pt} v^2 } \right)^2 }}.
\]
Here considering  $\kappa _2  = \frac{1}{{\rho _2 }}$ and equation
(\ref{A:5.4}) completes the proof.
\end{proof}\\
The relation between Gaussian curvature and the distribution
parameter of $M_2$ given by equation (\ref{A:5.10}) is called
\textbf{Lorentzian Lamarle formula} for the timelike
non-cylindrical
ruled surface with spacelike base curve and timelike ruling.\\
The Lamarle formula for the timelike ruled surface in
$\mathbb{R}_1^3$ is non--negative. So, we give the following
corollary.
\begin{corollary}\label{C:5.1}
Let $P$ be a distribution parameter and $K$ be a Gaussian
curvature of a timelike non-cylindrical ruled surface $M_2$ with
spacelike base curve and timelike ruling in $\mathbb{R}_1^3$. In
this case
\begin{enumerate}
    \item Along ruling as $v \to  \mp \infty$, $ K\left( {u,v} \right) \to
    0$.
    \item $K\left( {u,v} \right)=0$ if and only if $P=0$.
    \item If the distribution parameter of $M_2$ never vanishes, then $K\left( {u,v} \right)$ is
    continuous and as $v=0$ i.e. at the central point on each ruling, $K\left( {u,v} \right)$
    takes its minimum value.
\end{enumerate}
\end{corollary}

\begin{example}\label{E:5.1}
In $3-$dimensional Lorentz space $\mathbb{R}_1^3$.
\[
\varphi \left( {u,v} \right) = \left( { - v\sinh u\,,\,u\,,\, -
v\cosh u} \right)
\]
is a $3^{rd}$ type helicoid and a timelike non-cylindrical
ruled surface with spacelike base curve and timelike ruling, see: Figure 5.1.\\

\hfil\scalebox{1}{\includegraphics{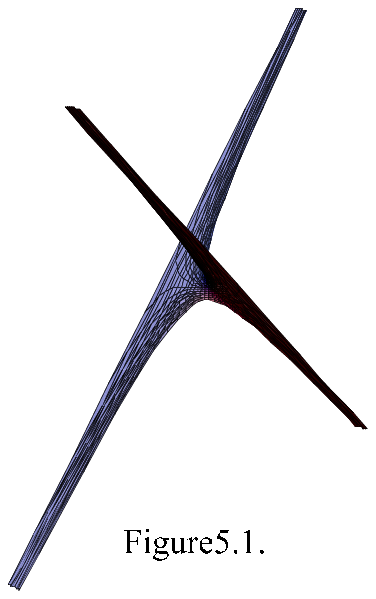}}\hfil
\hfil\scalebox{1}{\includegraphics{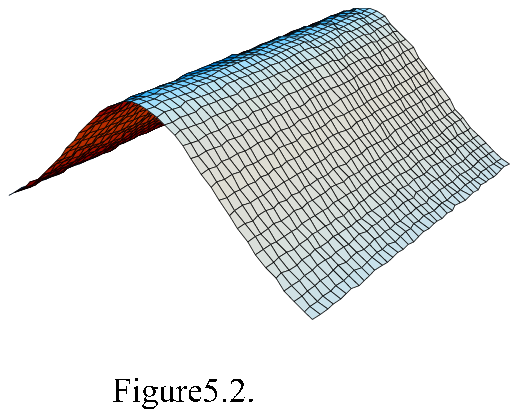}}\hfil
\\The Gaussian curvature of this $3^{rd}$ type helicoid is $K =  - \frac{1}{{\left( {1 + v^2 } \right)^2 }}$, see Figure 5.2.
\end{example}

\section{Lamarle Formula for Timelike Ruled Surface with Timelike Base Curve and Spacelike Ruling}
Suppose that the timelike ruled surface $M_3$ with timelike base
curve and spacelike ruling in three dimensional Lorentz space
$\mathbb{R}_1^3$ is parametrized as follows
\[
\begin{array}{l}
 \varphi _3 :{\rm I} \times \mathbb{R} \to \mathbb{R}_1^3  \\
 {\rm       }\left( {u,v} \right) \to \varphi _3 \left( {u,v} \right) = \vec \alpha _3 \left( u \right) + v\;\vec e_3 \left( u \right). \\
 \end{array}
\]
Considering that $ \left\| {\vec e_3 } \right\| = 1$ , $\vec n_3 =
\frac{{\vec e'_3 }}{{\left\| {\vec e'_3 } \right\|}}$ and $\vec \xi
_3  = \frac{{\vec e_3  \wedge \vec e'_3 }}{{\left\| {\vec e_3 \wedge
\vec e'_3 } \right\|}}$, we reach the orthonormal frame field
$\left\{ {\vec e_3 ,\vec n_3 ,\vec \xi _3 } \right\}$.This forms a
right handed system which is in type ${\rm \{ space}{\rm ,
space}{\rm , time\} }$. Thus we write

\begin{equation}\label{A:6.1}
\begin{array}{l}
 \left\langle {\vec e_3 ,\vec e_3 } \right\rangle  = \left\langle {\vec n_3 ,\vec n_3 } \right\rangle  =  - \left\langle {\vec \xi _3 ,\vec \xi _3 } \right\rangle  = 1 \\
 \left\langle {\vec e_3 ,\vec n_3 } \right\rangle  = \left\langle {\vec n_3 ,\vec \xi _3 } \right\rangle  = \left\langle {\vec \xi _3 ,\vec e_3 } \right\rangle  = 0 \\
 \end{array}
\end{equation}
and cross product is defined to be

\begin{equation}\label{A:6.2}
\vec e_3  \wedge \vec n_3  = \vec \xi _3 \quad ,\quad \vec n_3
\wedge \vec \xi _3  =- \vec e_3 \quad ,\quad \vec \xi _3  \wedge
\vec e_3  = -\vec n_3 .
\end{equation}
Differential formulae for this orthonormal system is expressed by

\begin{equation}\label{A:6.3}
\vec e'_3  = \kappa _3 \,\vec n_3 \quad ,\quad \vec n'_3  = -\kappa
_3 \,\vec e_3  + \tau _3 \,\vec \xi _3 \quad ,\quad \vec \xi' _3 =
\tau _3 \,\vec n_3 .
\end{equation}
Let the striction curve of timelike ruled surface given by equation
(\ref{A:3.2}) be  $\vec \beta _3 \left( u \right)$. This curve is a
timelike curve and the tangent vector of this curve stays within the
timelike plane $\left( {\vec e_3 ,\vec \xi _3} \right)$. Adopting
the hyperbolic angle $\sigma_3$ to be the angle between $\vec \beta
'_3$ and $\vec e_3$ we may write
\[
\vec \beta '_3  = \,\sinh \sigma _3 \,\vec e_3  + \,\cosh \sigma
_3 \,\vec \xi _3
\]
yielding the striction curve of $M_3$ to be
\[
\vec \beta _3  = \int {\left( {\sinh \sigma _3 \,\vec e_3  +
\,\cosh \sigma _3 \,\vec \xi _3 } \right)} \,du.
\]
The timelike non-cylindrical ruled surface $M_3$ with timelike
base curve and spacelike ruling is reparametrized by
\[
\tilde \varphi _3 \left( {u,v} \right) = \int {\left( {\,\sinh
\sigma _3 \,\vec e_3  + \,\cosh \sigma _3 \,\vec \xi _3 } \right)}
\,du + v\,\vec e_3 .
\]
The distribution parameter of this ruled surface is found to be
\[
P = \frac{{\det \left( {\,\sinh \sigma _3 \,\vec e_3  + \,\cosh
\sigma _3 \,\vec \xi _3 ,\vec e_3 ,\kappa _3 \,\vec n_3 }
\right)}}{{\left\langle {\kappa _3 \,\vec n_3 ,\kappa _3 \,\vec
n_3 } \right\rangle }} = \frac{{\cosh \sigma _3 }}{{\kappa _3 }}.
\]
The distribution parameter of $M_3$ becomes

\begin{equation}\label{A:6.4}
P = \rho _3 \,\cosh \sigma _3
\end{equation}
where $\kappa _3  = \frac{1}{{\rho _3 }}$. Taking equation
(\ref{A:6.2}) into consideration, we find from equation
(\ref{A:3.6}) that the unit normal vector of timelike
non-cylindrical ruled surface $M_3$ is

\begin{equation}\label{A:6.5}
\vec \eta _3  = \frac{{ - \,\cosh \sigma _3 \,\vec n_3  - {\kern
1pt} v{\kern 1pt} \kappa _3 \,\vec \xi _3 }}{{\sqrt {\left| {\cosh
^2 \sigma _3  - v^2 \kappa _3^2 } \right|} }}.
\end{equation}
Since $\kappa _3  = \frac{1}{{\rho _3 }}$, from equation
(\ref{A:6.4}) we find

\begin{equation}\label{A:6.6}
\vec \eta _3  =  - \frac{{{\kern 1pt} P\vec n_3  + {\kern 1pt}
v\vec \xi _3 }}{{\sqrt {\left| {P^2  - v^2 } \right|} }}.
\end{equation}
In addition to these, since the unit normal vector $\vec \eta _3$ of
timelike ruled surface $M_3$ is spacelike, that is, $P^2  - v^2
> 0$ i.e. $\left| P \right| > \left| v \right|$.\\
The partial differentials of $M_3$ with respect to $u $ and $v$
(from equation (\ref{A:6.3})) becomes

\begin{equation}\label{A:6.7}
\begin{array}{l}
 \tilde \varphi _{3u}  = \,\sinh \sigma _3 \,\vec e_3  + \,\cosh \sigma _3 \,\vec \xi _3  + v{\kern 1pt} \kappa _3 \,\vec n_3 {\kern 1pt} , \\
 \tilde \varphi _{3v}  = \vec e_3 . \\
 \end{array}
\end{equation}
Considering the last equations together with equation
(\ref{A:6.1}) the coefficients of first fundamental form of $M_3$
are

\begin{equation}\label{A:6.8}
\begin{array}{l}
 E = \left\langle {\tilde \varphi _{3u} ,\tilde \varphi _{3u} } \right\rangle  = \sinh ^2 \sigma _3  - \cosh ^2 \sigma _3  + v^2 \,\kappa _3^2  =  - 1 + v^2 \,\kappa _3^2 , \\
 F = \left\langle {\tilde \varphi _{3u} ,\tilde \varphi _{3v} } \right\rangle  = \sinh \sigma _3 , \\
 G = \left\langle {\tilde \varphi _{3v} ,\tilde \varphi _{3v} } \right\rangle  = 1. \\
 \end{array}
\end{equation}
Furthermore, considering equation (\ref{A:6.7}) we find for the
second order partial differentials of $M_3$ as

\[
\begin{array}{l}
 \tilde \varphi _{3uu}  = e_3 \left( {\sigma '_3 \,\cosh \sigma _3  - v{\kern 1pt} \kappa _3^2 } \right) + n_3 \left( {\kappa _3 {\kern 1pt} \sinh \sigma _3  + \tau _3 {\kern 1pt} \cosh \sigma _3  + v{\kern 1pt} \kappa '_3 } \right) + \xi _3 \left( {\sigma '_3 {\kern 1pt} \sinh \sigma _3  + v{\kern 1pt} \kappa _3 {\kern 1pt} \tau _3 } \right), \\
 \tilde \varphi _{3uv}  = n_3 {\kern 1pt} \kappa _3 , \\
 \varphi _{3vv}  = 0. \\
 \end{array}
\]
From equation (\ref{A:6.5}) and the last equations, the coefficients
of the second order principal form read to be

\begin{equation}\label{A:6.9}
\begin{array}{l}
 L = \left\langle {\tilde \varphi _{3uu} ,\vec \eta _3 } \right\rangle  = \frac{{ - \kappa _3 {\kern 1pt} \sinh \sigma _3 {\kern 1pt} \cosh \sigma _3  - \tau _3 {\kern 1pt} \cosh ^2 \sigma _3  - v{\kern 1pt} \kappa '_3 {\kern 1pt} \cosh \sigma _3  + v{\kern 1pt} \kappa _3 {\kern 1pt} \sigma '_3 {\kern 1pt} \sinh \sigma _3  + v^2 {\kern 1pt} \kappa _3^2 {\kern 1pt} \tau _3 }}{{\sqrt {\left| {\cosh ^2 \sigma _3  - v^2 \kappa _3^2 } \right|} }}, \\
 N = \left\langle {\tilde \varphi _{3uv} ,\vec \eta _3 } \right\rangle  = \frac{{\kappa _3 \cosh \sigma _3 }}{{\sqrt {\left| {\cosh ^2 \sigma _3  - v{\kern 1pt} \kappa _3^2 } \right|} }}, \\
 M = \left\langle {\tilde \varphi _{3vv} ,\vec \eta _3 } \right\rangle  = 0. \\
 \end{array}
\end{equation}
Taking equations (\ref{A:6.8}) and (\ref{A:6.9}) together into
account, we can give the following theorem for the Gaussian
curvature of timelike ruled surface $M_3$.

\begin{theorem}\label{T:6.1}
Let $M_3$ be a timelike non-cylindrical ruled surface with
timelike base curve and spacelike ruling in $\mathbb{R}_1^3$.
Adopting that the distribution parameter of $M_3$ is $P$, the
Gaussian curvature of $M_3$ becomes

\begin{equation}\label{A:6.10}
K = \frac{{P^2 }}{{\left( {P^2  - v^2 } \right)^2 }}
\end{equation}
where $P^2  - v^2  > 0$.
\end{theorem}

\begin{proof}
Substituting equations (\ref{A:6.8}) and (\ref{A:6.9}) into
equation (\ref{A:3.5}) we find the Gaussian curvature of $M_3$ to
be
\[
K = \frac{{\kappa _3^2 \,\cosh ^2 \sigma _3 }}{{\left( {\cosh ^2
\sigma _3  + \kappa _3^2 {\kern 1pt} v^2 } \right)^2 }}.
\]
Considering equation (\ref{A:6.4}) together with $\kappa
_3=\frac{1}{{\rho _3}}$ completes the proof.
\end{proof}\\

The relation between the Gaussian curvature and the distribution
parameter of timelike ruled surface given by equation (\ref{A:6.10})
is called \textbf{Lorentzian Lamarle formula} for
the timelike non-cylindrical ruled surface with timelike base curve and spacelike ruling.\\
The Lamarle formula for the timelike ruled surface in
$\mathbb{R}_1^3$ is non--negative. So, we give the following
corollary.

\begin{corollary}\label{C:6.1}
Let $M_3$ be a timelike non-cylindrical ruled surface with
timelike base curve and spacelike ruling in $\mathbb{R}_1^3$.
Considering that $P$ is the distribution parameter and $K$ is the
Gaussian curvature we see that

\begin{enumerate}
    \item Along ruling $v \to  \mp \infty$ the Gaussian curvature $ K\left( {u,v} \right) \to
    0$.
    \item $K\left( {u,v} \right)=0$ if and only if $P=0$.
    \item If the distribution parameter of $M_3$ never vanishes, then $K\left( {u,v} \right)$ is
    continuous and as $v=0$ i.e. at the central point on each ruling, $K\left( {u,v} \right)$
    takes its minimum value.
\end{enumerate}
\end{corollary}

\begin{example}\label{E:6.1}
Let us parametrize a $1^{st}$ type helicoid as
\[
\varphi \left( {u,v} \right) = \left( { - v\cos u\,\,,\, - v\sin
u,\,u} \right)
\]
which is timelike non-cylindrical ruled surface with timelike base
curve and spacelike ruling in $3-$dimensional Lorentz space
$\mathbb{R}_1^3$ and here $ - 1 < v < 1$, see: Figure 6.1.\\

\hfil\scalebox{1}{\includegraphics{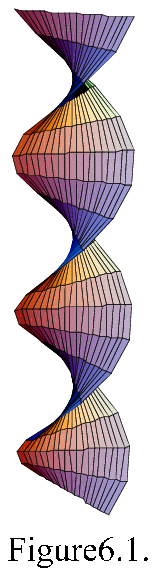}}\hfil
\hfil\scalebox{1}{\includegraphics{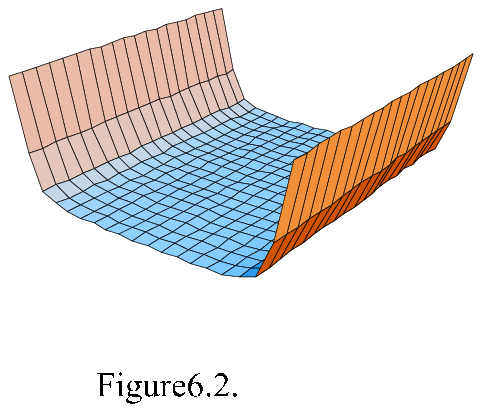}}\hfil
\\The Gaussian curvature of this $1^{st}$ type helicoid is $K = \frac{1}{{\left( {1 - v^2 } \right)^2 }}\quad ,\quad \left| v
\right| < 1$, see Figure 6.2.
\end{example}

\begin{acknowledgement}
The authors are very grateful to Prof.Dr. Ibrahim Okur for
improving presentation of the paper.
\end{acknowledgement}

\end {document}